\documentclass[11pt]{amsart}
\usepackage[utf8]{inputenc}
\usepackage{amsmath}
\usepackage[makeroom]{cancel}
\usepackage{tikz-cd}
\usepackage{todonotes}
\usepackage{amsfonts, amsthm}
\usepackage{ytableau}
\usepackage{float}
\usepackage{subfig}
\usepackage{tabu}
\usepackage{mathtools}
\usepackage{amssymb, hyperref, color, graphicx, caption}
\usepackage{geometry}
\usepackage{mdframed}
\usepackage{hyperref}
\usepackage{tikzsymbols}
\usepackage{tcolorbox}
\usepackage{xcolor}   
\usepackage{hyperref}
\hypersetup{
    colorlinks=true, 
    linktoc=all,     
    linkcolor=blue,  
    citecolor=blue
}
\newcommand{\Z}{\mathbb{Z}}

\newcommand{\C}{\mathbb{C}}

\newcommand{\R}{\mathbb{R}}

\newcommand{\Q}{\mathbb{Q}}

\newtheorem{thm}{Theorem}

\numberwithin{thm}{section}

\newtheorem{cor}[thm]{Corollary}
\newtheorem{lem}[thm]{Lemma}
\newtheorem{conj}[thm]{Conjecture}
\newtheorem{prop}[thm]{Proposition}
\theoremstyle{definition}
\newtheorem{defn}[thm]{Definition}

\newtheorem*{xrem}{Remark} 

\newtheorem*{exstar}{Example}

\makeatletter
\newtheorem*{rep@theorem}{\rep@title}
\newcommand{\newreptheorem}[2]{%
\newenvironment{rep#1}[1]{%
 \def\rep@title{#2 \ref{##1}}%
 \begin{rep@theorem}}%
 {\end{rep@theorem}}}
\makeatother

\newreptheorem{theorem}{Theorem}

\newreptheorem{lemma}{Lemma}

\begin{document}

\title[Hook Lengths Divisible by Two or Three]{Distributions of Hook Lengths Divisible by Two or Three}

\subjclass[2020]{}
\keywords{Partitions, hook lengths, symmetric groups}

\author[H.~Lang]{Hannah Lang}
\address{Department of Mathematics \\
Harvard University \\
1 Oxford Street \\
Cambridge, MA 02138}
\email{\href{mailto:hlang@college.harvard.edu}{hlang@college.harvard.edu}}

\author[H.~Wan]{Hamilton Wan}
\address{Department of Mathematics \\
Yale University\\
10 Hillhouse Avenue\\
New Haven, CT 06520}
\email{\href{mailto:hamilton.wan@yale.edu}{hamilton.wan@yale.edu}}

\author[N.~Xu]{Nancy Xu}
\address{Department of Mathematics \\
Princeton University \\
304 Washington Road \\
Princeton, NJ 08544}
\email{\href{mailto:nancyx@princeton.edu}{nancyx@princeton.edu}}

\maketitle

\begin{abstract}
    For fixed $t = 2$ or $3$, we investigate the statistical properties of $\{\hat{Y}_t(n)\}$, the sequence of random variables corresponding to the number of hook lengths divisible by $t$ among the partitions of $n$. We characterize the support of $\hat{Y}_t(n)$ and show, in accordance with empirical observations, that the support is vanishingly small for large $n$. Moreover, we demonstrate that the nonzero values of the mass functions of $\hat{Y}_2(n)$ and $\hat{Y}_3(n)$ approximate continuous functions. Finally, we prove that although the mass functions fail to converge, the cumulative distribution functions of $\{\hat{Y}_2(n)\}$ and $\{\hat{Y}_3(n)\}$ converge pointwise to shifted Gamma distributions, completing a characterization initiated by Griffin--Ono--Tsai for $t \geq 4$. 
\end{abstract}

\section{Introduction}\label{sec:intro}

We say that a sequence of nonincreasing positive integers $\lambda = (\lambda_1, \hdots, \lambda_m)$ is a \textit{partition} of a positive integer $n$ if $\lambda_1 + \cdots + \lambda_m = n$. The \textit{Young diagram} of the partition $\lambda = (\lambda_1,\ldots,\lambda_m)$ is a left-justified array of square cells with rows of length $\lambda_1, \hdots, \lambda_m$, and the \textit{hook length} of any cell in the Young diagram is the sum of the number of cells to the right of it in its row, the number of cells below it in its column, and one (to account for the cell itself). 

\begin{exstar} Here we give the Young diagram for the partition $\lambda = (6,4,3,1)$ of $n = 14$, where each cell is labelled with its corresponding hook length.
\begin{figure}[H]
    \centering
    \ytableausetup{centertableaux}
    \begin{ytableau}
    9 & 7 & 6 & 4 & 2 & 1 \\
    6 & 4 & 3 & 1 \\
    4 & 2 & 1 \\
    1
    \end{ytableau}
    \label{fig:young_diagram_example}
\end{figure}
\end{exstar}

Partition hook lengths have many applications in combinatorics, number theory, and representation theory. For instance, consider the famous correspondence between complex finite-dimensional irreducible representations of the symmetric group $S_n$ and the Young diagrams of the partitions of $n$. For any prime $p$, the Young diagrams associated with the set of $p$-core partitions, that is, partitions with no hook lengths divisible by $p$, correspond precisely to those irreducible representations that remain irreducible upon reduction modulo $p$ \cite[Chapter 6]{JK84}. Moreover, the number of standard Young tableaux of the partition $\lambda$, a quantity of interest in algebraic combinatorics, is given by the Frame-Robinson-Thrall hook length formula \cite{FRT54}
\[
d_{\lambda} = \frac{n!}{\prod_{h \in \mathcal{H}(\lambda)}h},
\] 
where $\mathcal{H}(\lambda)$ denotes the multiset of hook lengths associated to a partition $\lambda$. The value $d_{\lambda}$ is also the dimension of the irreducible representation of $S_n$ corresponding to $\lambda$.

In addition, hook lengths arise naturally in mathematical physics and the study of modular forms. For instance, Nekrasov--Okounkov \cite{NO06} derived the identity 
\[
\sum_\lambda q^{\lambda} \prod_{h \in \mathcal{H}(\lambda)} \left( 1 -\frac{z}{h^2} \right) = \prod_{n=1}^\infty (1-q^n)^{z-1}
\] in their study of the Seiberg-Witten theory, which expresses the powers of the Euler product in terms of $\mathcal{H}(\lambda)$. Griffin, Ono, and Tsai studied the statistical properties of $\mathcal{H}(\lambda)$ in \cite{GOT22}. In particular, they showed that for fixed $t$, the distribution of the random variables $\{t(n)\}$ corresponding to the number of hooks of length exactly $t$ among partitions of $n$ is asymptotically normal.

Using $\mathcal{H}_t(\lambda)$, the multiset of hook lengths of $\lambda$ divisible by $t$, Han \cite[Theorem 1.3]{Han10} derived the following generalization of the Nekrasov--Okounkov identity:  
\begin{align}\label{eqn:han_generalization} \sum_{\lambda } x^{|\lambda|}\prod_{h \in \mathcal{H}_t(\lambda)} \left(y- \frac{tyz}{h^2}\right) = \prod_{k \geq 1} \frac{(1-x^{tk})^t}{(1-(yx^t)^k)^{t-z}(1-x^k)}. \end{align} 

In view of the importance of $\mathcal{H}_{t}(
\lambda)$, we study the sequence of random variables $\{\hat{Y}_t(n)\}$ for a fixed integer $t \geq 2$, where the values of $\hat{Y}_t(n)$ are the number of hook lengths divisible by $t$ in the partitions of $n$. 
For the sake of convenience, we define the polynomial
\begin{align}
P_t(n,x) := \sum_{\lambda \vdash n} x^{\#\mathcal{H}_t(\lambda)} = \sum_{m=0}^\infty p_t(m,n)x^m
\end{align} for any positive integers $t$ and $n$, where $p_t(m,n)$ denotes the number of partitions of $n$ with $m$ hooks of length divisible by $t$. Recently, Griffin--Ono--Tsai \cite{GOT22} showed for the case $t \geq 4$ that the cumulative distribution functions of $\hat{Y}_t(n)$ converge pointwise to a shifted Gamma distribution with shape $k(t) := (t - 1)/2$ and scale $\theta(t) := \sqrt{2/(t - 1)}$ by estimating the polynomial $P_t(n,x)$ at specific values of $x$. To be precise, they proved that for $t \geq 4$, the random variables $\{\hat{Y}_t(n)\}$ satisfy
\[
\frac{n\pi}{\sqrt{3(t-1)n}} - \frac{t\pi}{\sqrt{3(t-1)n}}\hat{Y}_t(n) \sim X_{k(t), \theta(t)},\]
where $X_{k(t), \theta(t)}$ is a random variable satisfying the Gamma distribution with shape $k(t)$ and scale $\theta(t)$ that we defined earlier. 

Ono asked for a resolution to the analogous question for the cases $t = 2$ and $3$. The methods in \cite{GOT22} do not apply to the cases $t=2$ and $3$ because they rely on the existence of moment generating functions for Gamma distributions with shape parameter $k(t) > 1$. Moreover, Griffin--Ono--Tsai observed that for $t=2$ and $3$, the support of $\hat{Y}_t(n)$ is very sparse. For example, consider the distribution of nonzero coefficients in the polynomial
\begin{align*}
P_2(100, x) = \sum_{\lambda \vdash 100} x^{\#\mathcal{H}_2(\lambda)} &= 752 x^{11} + 8470 x^{17} + 1046705 x^{32}  +3157789 x^{36} \\&\quad+ 31551450 x^{45} + 
 51124970 x^{47} + 103679156 x^{50},
\end{align*} also depicted in Figure \ref{fig:t2_sparse}.

Due to the frequency of zeros, the probability mass functions do not appear to converge to a continuous probability density function, which further distinguishes these cases from the case of $t \geq 4$ (see, for instance, Figure \ref{fig:conjecturet11} in Section \ref{sec:conj}).

\begin{figure}[h!]
    \centering
    \includegraphics[scale=0.8]{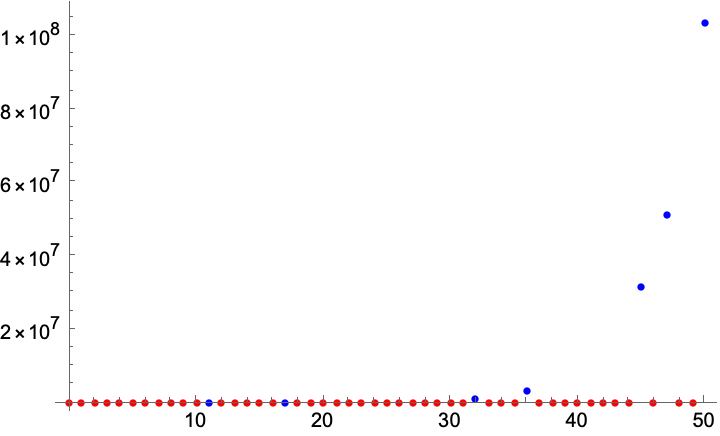}
    \caption{Coefficients of the polynomial $P_2(100,x)$. Observe the sparse support (indicated in blue) and the frequency of zeros (indicated in red).}
    \label{fig:t2_sparse}
\end{figure}

Motivated by the result for $t \geq 4$, we apply a change of variables and define the probability mass functions $\widetilde{f}_{t;n}(x)$ corresponding to 
\begin{align}\widetilde{Y}_t(n) := \frac{n\pi}{\sqrt{3(t-1)n}} - \frac{t\pi}{\sqrt{3(t-1)n}}\hat{Y}_t(n)\end{align}
for any positive integer $n$. 
Similarly, let $g_t(x)$ denote the probability density function for the random variable $X_{k(t), \theta(t)}$. Our first goal is to compare the behavior of the \textit{scaled} mass functions (to account for the change of variables) \begin{align}f_{t;n}(x) := \frac{\sqrt{3(t-1)n}}{t\pi} \; \widetilde{f}_{t;n}(x)\end{align} with $g_t(x)$ in the cases $t = 2$ and $3$. Moreover, for the sake of brevity, we define for each positive integer $n$ \[\mathcal{S}_{t;n} := \left\{x \in \R \ \bigg|\ \frac{n}{t} - \frac{\sqrt{3(t-1)n}}{t\pi}x \in \{0, 1,\ldots,\lfloor n/t\rfloor\}\right\},\] the set of $x \in \R$ where the function $f_{t;n}$ is defined. 

We begin by providing an explicit description of the supports of $\hat{Y}_t(n)$ and show that almost all of the coefficients of $P_t(n,x)$ are zero for $t = 2$ and $3$. To do so, we derive a formula for the coefficients $p_t(m,n)$ related to the number of $t$-colored partitions of $m$. In the case $t = 2$, we relate the vanishing of these coefficients to the count of triangular numbers at most $n/2$, and in the case $t = 3$, we obtain our results by studying the multiplicative properties of the Fourier coefficients of a cusp form with complex multiplication. Using our explicit description, we show that, in some sense, the sparse support of the mass functions $f_{t;n}$ implies that they cannot converge to $g_t(x)$ for $t = 2,3$. 

\begin{thm}\label{intro_thm:support_t2}
If $t=2$, then the following are true:
\begin{itemize}
    \item[(a)] We have $p_2(m,n) \neq 0$ if and only if $n - 2m$ is a triangular number.
    \item[(b)] As $n \to \infty$, \[\frac{\#\{0 \leq m \leq \deg P_2(n,x)\ |\ p_2(m,n) \neq 0\}}{\deg P_{2}(n,x)} = O(n^{-1/2}).\]
    \item[(c)] If there exists $n_0$ such that $x \in \mathcal{S}_{2;n_0}$, then there exists an infinite sequence $\{n_j\}_j$ such that $x \in \mathcal{S}_{2;n_j}$ for all $j$, and the sequence $f_{2,n_j}(x)$ does not converge to $g_2(x)$ as $j \to \infty$. 
\end{itemize}
\end{thm}

\begin{thm}\label{intro_thm:support_t3}
If $t=3$, then the following are true:
\begin{itemize}
    \item[(a)] We have $p_3(m,n) \neq 0$ if and only if $\operatorname{ord}_r(3(n-3m)+1) \equiv 0 \pmod{2}$ for every prime $r \equiv 2 \pmod{3}$.
    \item[(b)] As $n \to \infty$, \[\frac{\#\{0 \leq m \leq \deg P_3(n,x)\ |\ p_3(m,n) \neq 0\}}{\deg P_{3}(n,x)} = O\left(\frac{1}{\sqrt{\log(n)}}\right).\]
    \item[(c)] If there exists $n_0$ such that $x \in \mathcal{S}_{3;n_0}$, then there exists an infinite sequence $\{n_j\}_j$ such that $x \in \mathcal{S}_{3; n_j}$ for all $j$, and the sequence $f_{3,n_j}(x)$ does not converge to $g_3(x)$ as $j \to \infty$. 
\end{itemize}
\end{thm}

Note that, since the points at which the functions $f_{t;n}$ are defined vary with $n$, it is not meaningful to discuss pointwise convergence to $g_t$ in the usual sense. Instead, our results in Theorems \ref{intro_thm:support_t2}(c) and \ref{intro_thm:support_t3}(c) provide an alternative, allowing us to describe the failure of $f_{t;n}$ to converge pointwise to $g_t$ for specific subsequences of $n$ where this notion makes sense. 

\begin{xrem} Theorems \ref{intro_thm:support_t2}(b) and \ref{intro_thm:support_t3}(b) imply that, for both $t = 2$ and $3$, the coefficients of $P_t(n,x)$ are zero on a subset with asymptotic density $1$. Namely, 
\[
\lim_{n \to \infty} \frac{\#\{ 0 \leq m \leq \deg P_t(n,x) \mid p_t(m,n) = 0\}}{\deg P_t(n,x)} = 1
\]
for both $t = 2$ and $3.$
\end{xrem}

\begin{exstar}
 The behavior described in the previous remark is illustrated in Figure \ref{fig:tableofnondiffablepoints}, which contains statistics on the proportion of nonzero coefficients of $P_2(n, x)$ and $P_3(n, x)$ with respect to the degrees of the polynomials up to $n = 5000$. Observe that the proportion of nonzero coefficients decreases faster for $t = 2$ than for $t = 3$, reflecting the bounds $O(n^{-1/2})$ for the former and $O(\log(n)^{-1/2})$ for the latter. Figure \ref{fig:t3_support_sparse} plots $\hat{Y}_3(n)$ for $n = 100, 500, 1000$, and $2000$.
\end{exstar}

\begin{center}
\begin{figure}[h!]
{\tabulinesep=1.2mm
\begin{tabu}{ |c|c|c| } 
 \hline
 $n$ & $\dfrac{\#\{p_2(m,n) \neq 0\}}{\deg P_2(n,x)}$ & $\dfrac{\#\{p_3(m,n) \neq 0\}}{\deg P_3(n,x)}$ \\
 \hline
 \hline
 100 & 0.14000\textcolor{white}{...} & 0.63636... \\
 \hline
 500 & 0.06400\textcolor{white}{...} & 0.50602... \\ 
 \hline
 1000 & 0.04600\textcolor{white}{...} & 0.47147...\\ 
 \hline
 1500 & 0.03600\textcolor{white}{...} & 0.46200\textcolor{white}{...}  \\ 
 \hline
 2000 & 0.03100\textcolor{white}{...} & 0.45496... \\
 \hline
 2500 & 0.02800\textcolor{white}{...} & 0.44658... \\ 
 \hline
 3000 & 0.02600\textcolor{white}{...} & 0.44400\textcolor{white}{...}  \\ 
 \hline
 3500 & 0.02400\textcolor{white}{...} & 0.43825... \\ 
 \hline
 4000 & 0.02250\textcolor{white}{...} & 0.43661... \\
 \hline
 4500 & 0.02088... & 0.43200\textcolor{white}{...} \\
 \hline
 5000 & 0.02000\textcolor{white}{...} & 0.43097... \\
 \hline
\end{tabu}}
    \caption{Table depicting the proportion of nonzero coefficients of $P_t(n,x)$ for $t = 2,3$ across various values of $n$.}
    \label{fig:tableofnondiffablepoints}
\end{figure}
\end{center}

\begin{figure}[h!]
\begin{minipage}{.5\linewidth}
\centering
\subfloat[Plot of $\hat{Y}_3(100)$ \label{subfig:y3100}]{\includegraphics[width=3in]{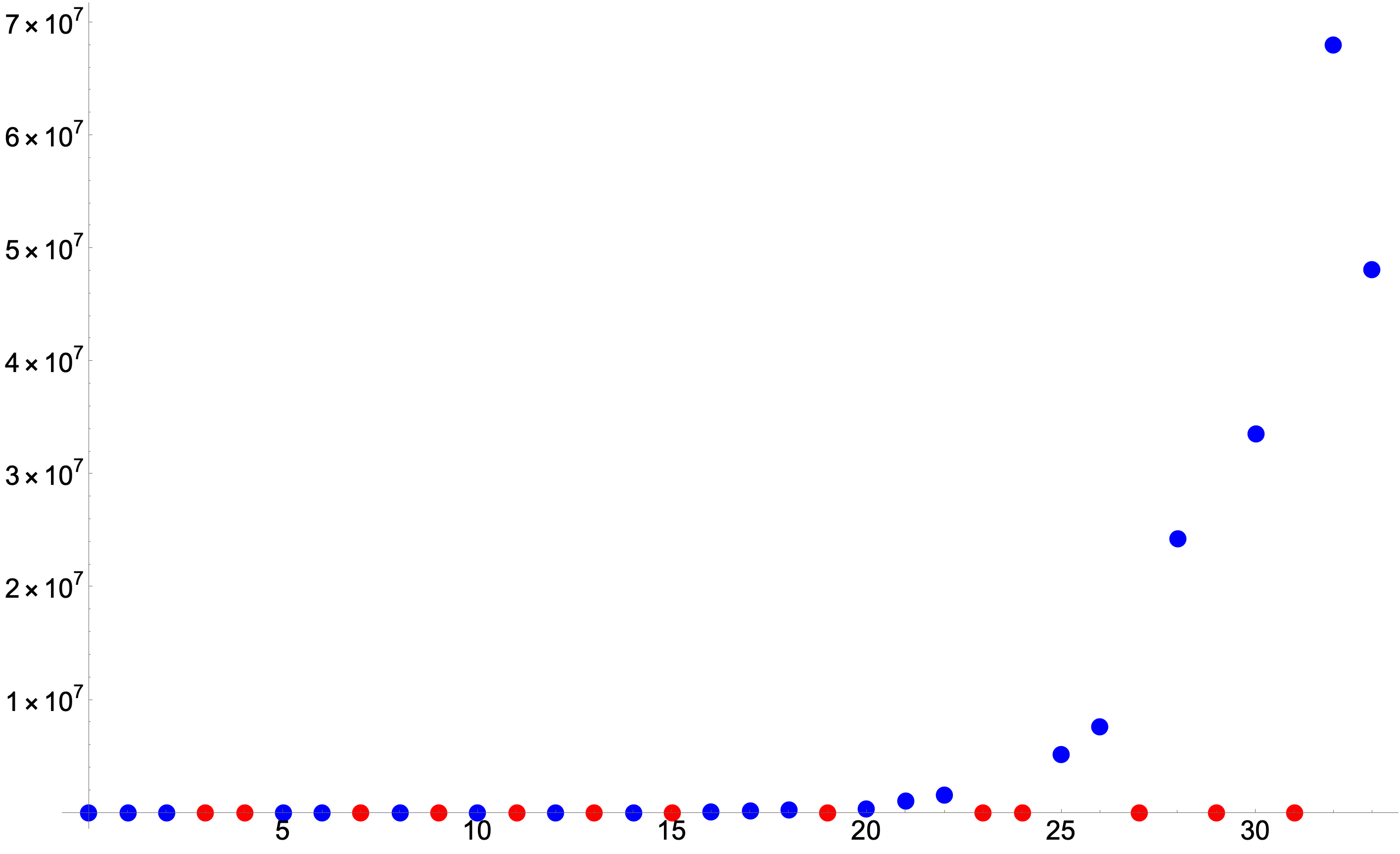}}
\end{minipage}%
\begin{minipage}{.5\linewidth}
\centering
\subfloat[Plot of $\hat{Y}_3(500)$.\label{subfig:y3500}]{\includegraphics[width=3in]{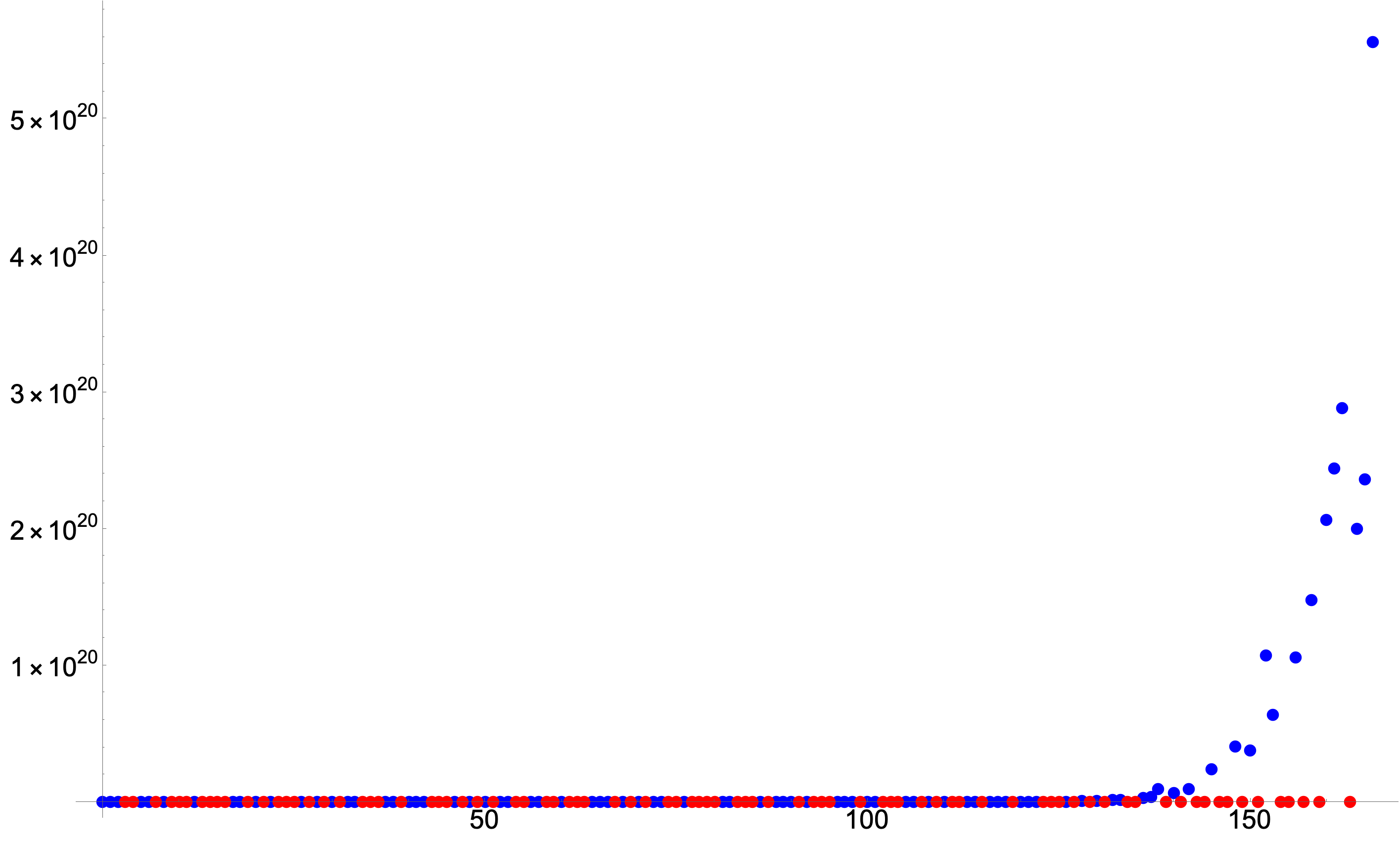}}
\end{minipage}%
\par\medskip
\centering
\begin{minipage}{.5\linewidth}
\centering
\subfloat[Plot of $\hat{Y}_3(1000)$.\label{subfig:y31000}]{\includegraphics[width=3in]{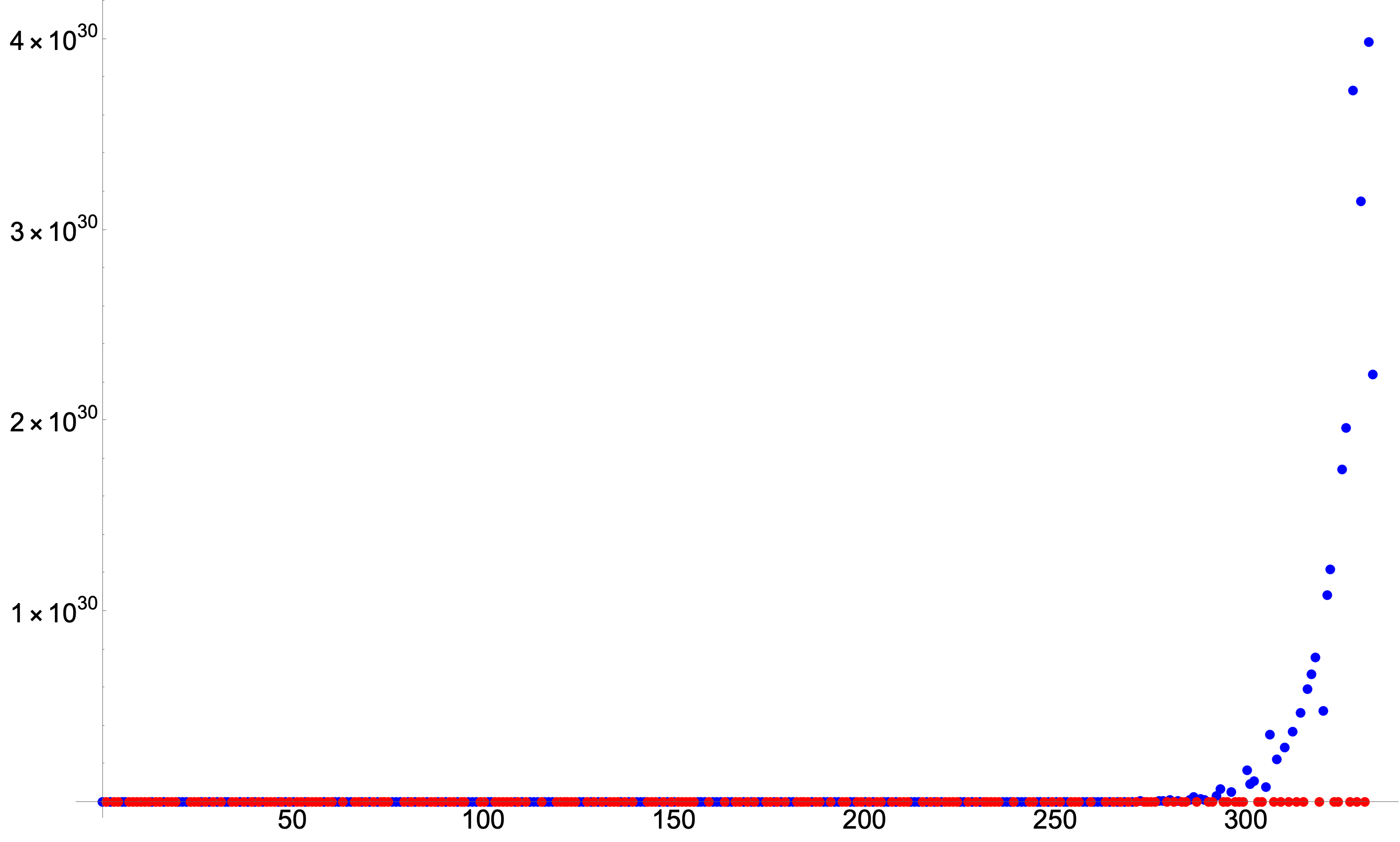}}
\end{minipage}%
\begin{minipage}{.5\linewidth}
\centering
\subfloat[Plot of $\hat{Y}_3(2000)$. \label{subfig:y32000}]{\includegraphics[width=3in]{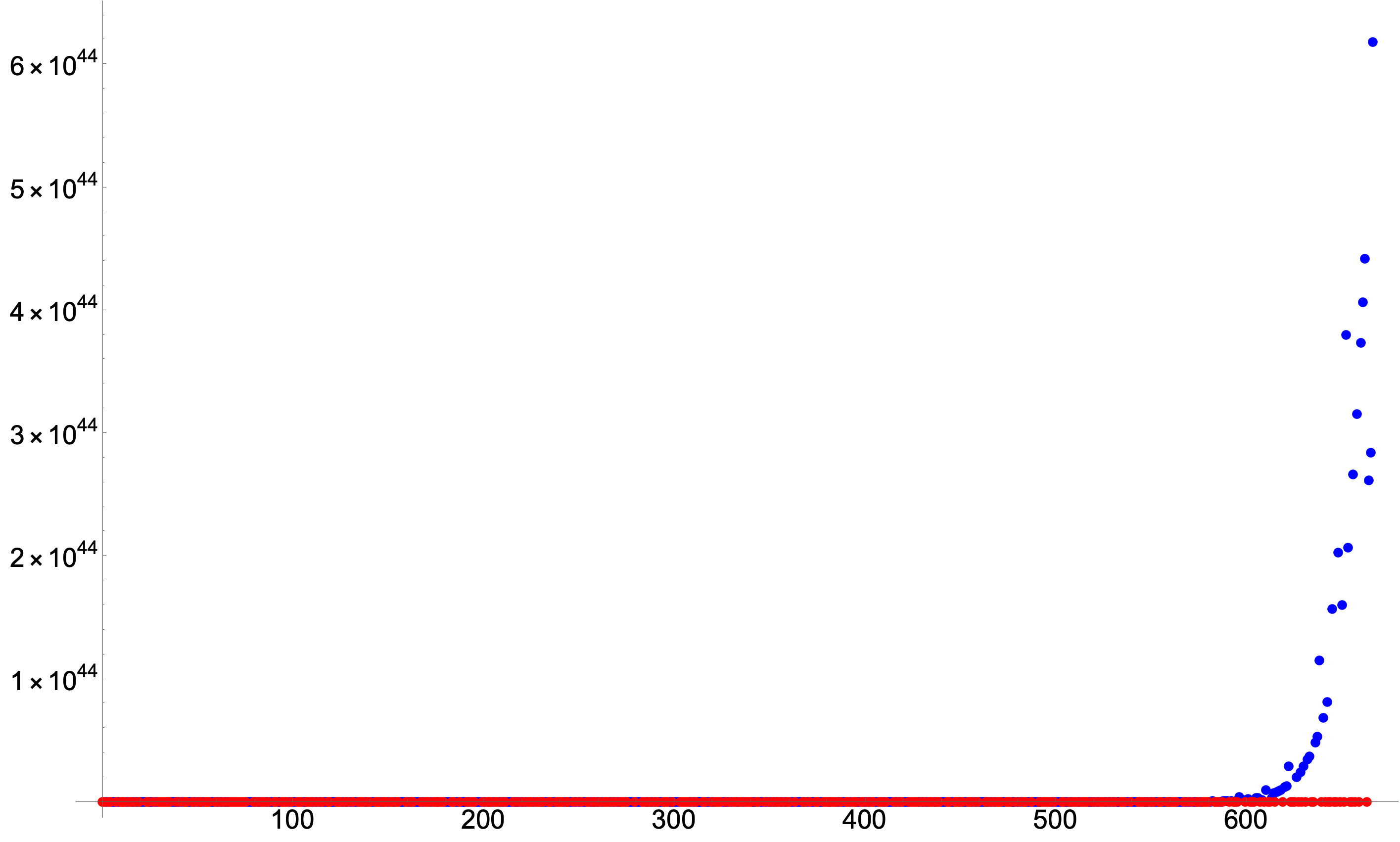}}
\end{minipage}

\caption{Observe that the support of $\hat{Y}_3(n)$ becomes sparser as $n$ grows larger, exemplifying the behavior outlined in Theorem \ref{intro_thm:support_t3}(b).}
\label{fig:t3_support_sparse}
\end{figure}

On the other hand, Figure \ref{fig:support_looks_smooth} suggests that the points in the support of $f_{2;n}(x)$ do in fact lie on a continuous curve, and Figure \ref{fig:t3_support} indicates that the points in the support of $f_{3;n}(x)$ approximate multiple continuous curves. Thus, a natural question to ask is to what extent $f_{t;n}$ converges to $g_t$ when we restrict the domain to the support. 

As we will see in Subsection \ref{sec:even_hl}, for $t=2$ the phenomenon depicted in Figure \ref{fig:support_looks_smooth} originates from the fact that the values of $f_{2;n}(x)$ are related to counts of two-colored partitions. Similarly, we discuss in Subsection \ref{sec:three_hl} how the relation between the values of $f_{3;n}(x)$ with three-colored partitions. We handle the case of $t=2$ in Theorem \ref{intro_thm:pmf_no_converge} and $t=3$ in Theorem \ref{intro_thm:pmf_converge_t3}. In particular, for $t=2$ we explicitly derive a formula for the apparent continuous curve for each $n$, but also show that these functions do not converge pointwise to $g_2(x)$. 

\begin{figure}[h!]
    \centering
    \includegraphics[scale=0.75]{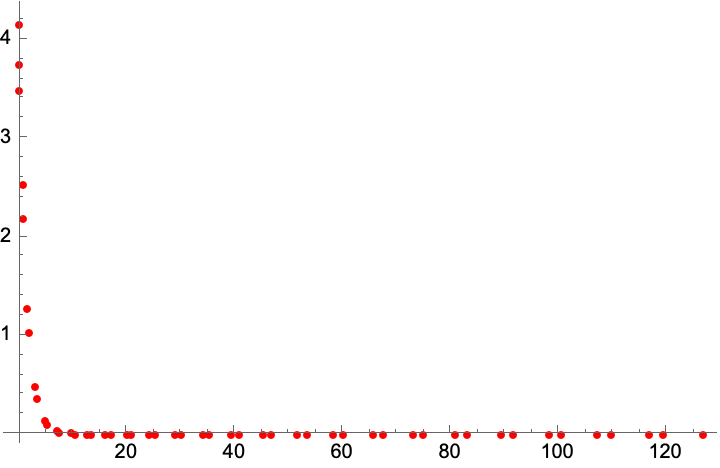}
    \caption{The values of $f_{2;5000}(x)$ in its support appear to approximate a continuous curve.}
    \label{fig:support_looks_smooth}
\end{figure}

\newpage
\begin{thm}\label{intro_thm:pmf_no_converge} 
Assuming the notation above, the following are true:
\begin{itemize}
    \item[(a)] For positive integers $n$, let \[
        h_{2;n}(x) := \frac{3^{1/4}n^{3/2} \exp \left(\pi \sqrt{\frac{2}{3}} \left( \sqrt{n - \frac{\sqrt{3n}}{\pi}x} - \sqrt{n}\right) \right)}{2 \pi\left( \frac{n}{2} - \frac{\sqrt{3n}}{2\pi}x \right)^{5/4}}.
    \] For all $x$ such that $f_{2;n}(x) \neq 0$, we have that $f_{2;n}(x) \sim h_{2;n}(x)(1+\beta_2(x,n))$ as $n \to \infty$, where $\beta_2(x,n)$ is a constant satisfying $\beta_2(x) = O_\varepsilon((n - x\sqrt{n})^{-1/4+\varepsilon})$ for any $\varepsilon > 0$.
    
    \item[(b)] For all $x \in \R$ and for any subsequence $\{n_j\}_{j=1}^\infty$, the sequence $h_{2;n_j}(x)$ does not converge to $g_2(x)$ as $j \to \infty$.
\end{itemize}
\end{thm}

\begin{figure}[h!]
    \centering
    \includegraphics[width=3.5in]{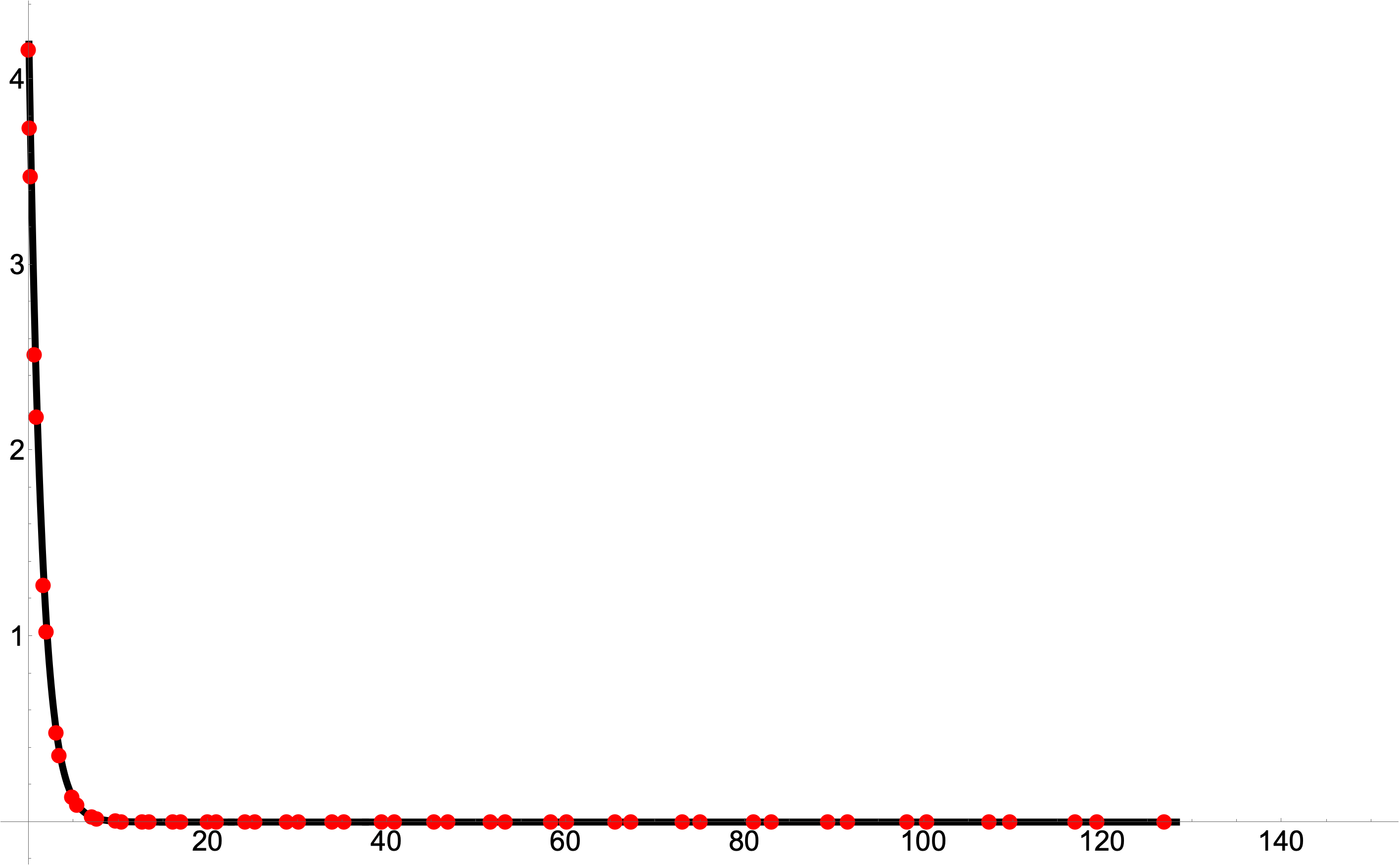}
    \caption{The values of $f_{2;5000}(x)$ in its support (in red) along with the continuous approximation $h_{2;5000}(x)$ (in black) derived in the proof of Theorem \ref{intro_thm:pmf_no_converge}.}
    \label{fig:support_with_curve}
\end{figure}

\begin{figure}[h!]
    \centering
    \includegraphics[width=4in]{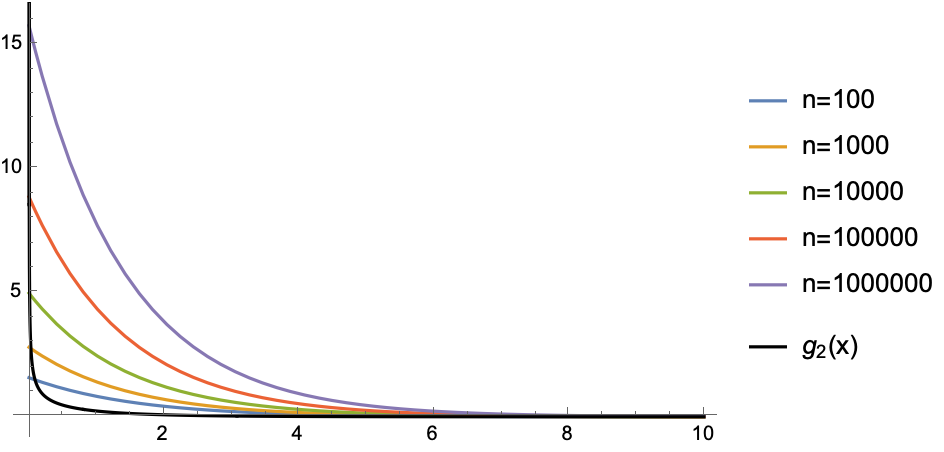}
    \caption{The approximations $h_{2;n}(x)$, depicted here for various values of $n$, do not converge pointwise to $g_2(x)$, depicted here in black, as $n \to \infty$, exemplifying Theorem \ref{intro_thm:pmf_no_converge}.}
    \label{fig:approximations_dont_converge}
\end{figure}

In the case of $t=3$, we prove that the nonzero points of $f_{3;n}(x)$ are asymptotic to integer multiples of a continuous function $h_{3;n}(x)$. We find an explicit formula for these continuous functions and show that they converge to scalar multiples of $g_3$. This result contrasts slightly with the $t=2$ case, but still falls short of outright convergence. 

\newpage
\begin{thm}\label{intro_thm:pmf_converge_t3} 
Assuming the notation above, the following are true:
\begin{itemize}
    \item[(a)] For positive integers $n$, let \[
        h_{3;n}(x) := \frac{3\sqrt{3}n^{3/2}\exp\left(\pi\sqrt{\frac{2}{3}}\left(\sqrt{n-\frac{\sqrt{6n}}{\pi}x} - \sqrt{n}\right)\right)}{2\pi\left(n - \frac{\sqrt{6n}}{\pi}x\right)^{3/2}}.
    \] 
    For all $x$ such that $f_{3;n}(x) \neq 0$, there exists a positive integer $\alpha(x,n)$ such that \[f_{3;n}(x) \sim \alpha(x, n)h_{3;n}(x)(1+\beta_3(x))\] as $n \to \infty$, where $\beta_3(x,n)$ is a constant satisfying $\beta_3(x,n) = O_\varepsilon((n - x\sqrt{n})^{-1/4+\varepsilon})$ for any $\varepsilon > 0$.
    
    \item[(b)] For all $x \in \R$, we have \[h_{3;n}(x) \to \dfrac{3\sqrt{3}}{2\pi}g_3(x)\] as $n \to \infty$.
\end{itemize}
\end{thm}

\begin{figure}[h!]
    \centering
    \includegraphics[width=4in]{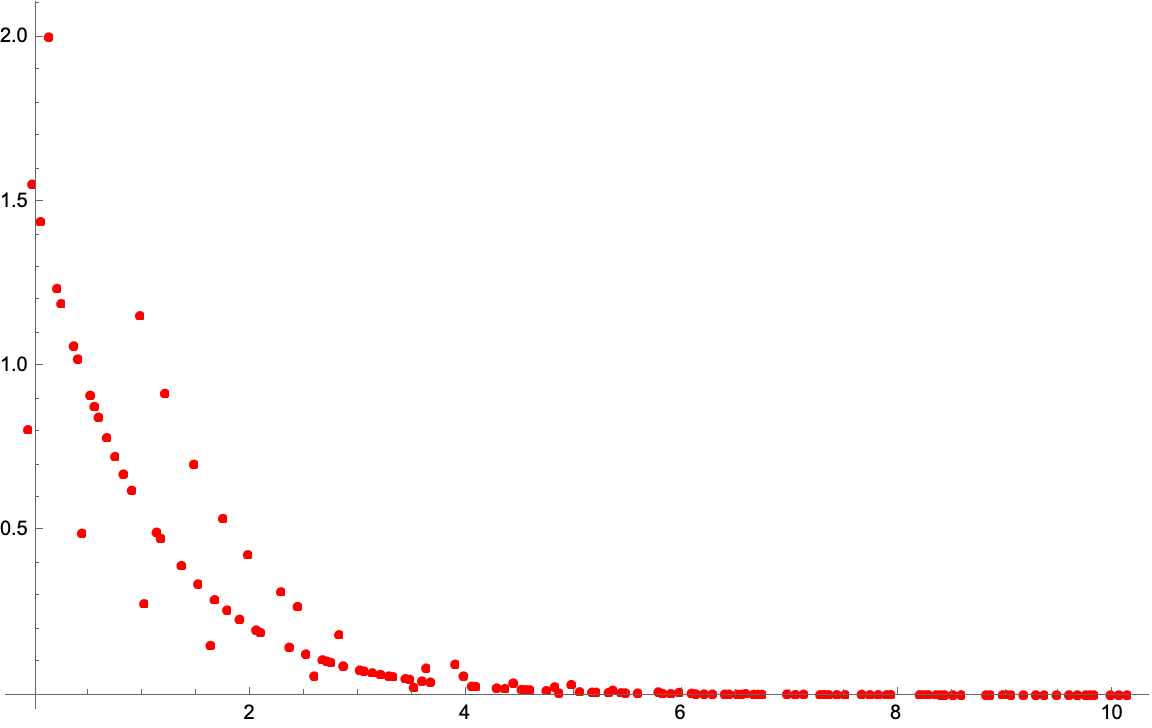}
    \caption{The values of $f_{3;10000}$ in its support appear to approximate several continuous curves.}
    \label{fig:t3_support}
\end{figure}

\begin{xrem}
In Proposition \ref{prop:explicit_coefficients_t_3}, we will show that the coefficients $p_3(m, n)$ can be written in terms of the character sum $c(k) = \sum_{d | 3k + 1} \left(\frac{d}{3}\right)$ and the number of $3$-colored partitions of $m$. The former is the source of the integer-valued scalar $\alpha(x, n)$ in Theorem \ref{intro_thm:pmf_converge_t3}.
\end{xrem}

\begin{figure}[h!]
    \centering
    \includegraphics[width=4in]{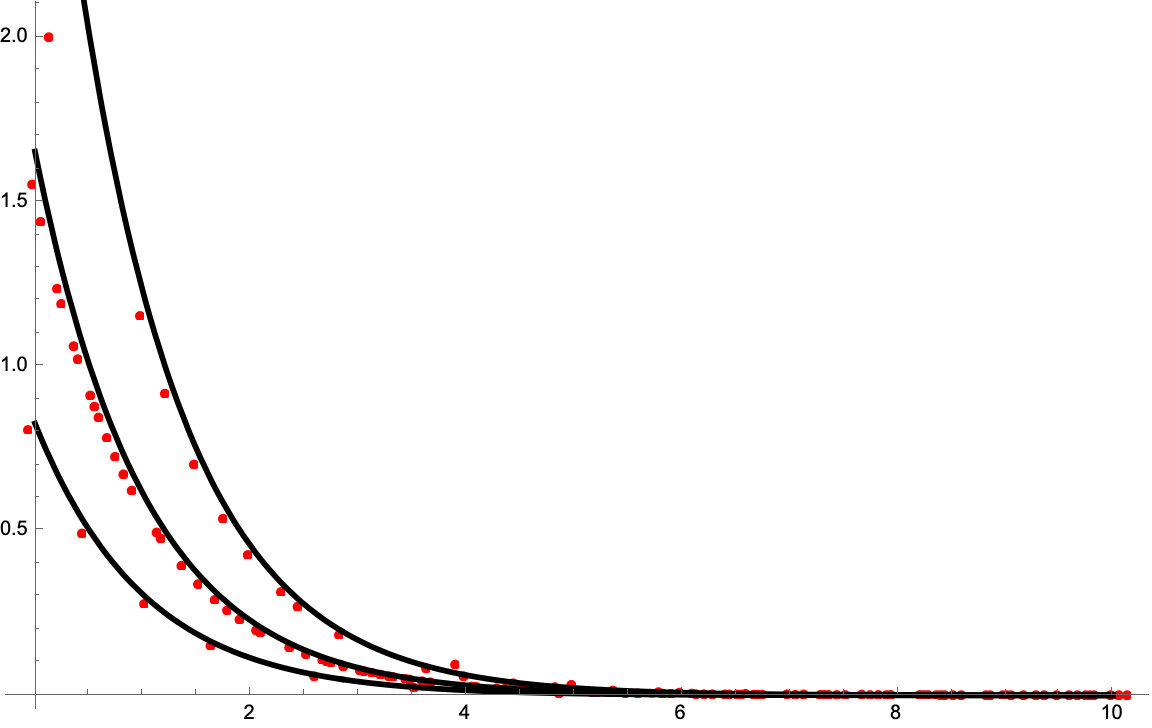}
    \caption{The values of $f_{2;10000}(x)$ in its support (in red) along with the continuous approximations $h_{3;10000}(x)$, $2h_{3;10000}(x)$, and $4h_{3,10000}(x)$ (in black) derived in the proof of Theorem \ref{intro_thm:pmf_converge_t3}.}
    \label{fig:t3_support_with_curve}
\end{figure}

Theorems \ref{intro_thm:support_t2}, \ref{intro_thm:support_t3}, and \ref{intro_thm:pmf_no_converge} illustrate how the probability mass functions of the random variables $\hat{Y}_{t}(n)$ fail to converge, in various senses, to the probability density functions consistent with the work of \cite{GOT22}, while Theorem \ref{intro_thm:pmf_converge_t3} gives convergence to scalar multiples of the probability density function of the Gamma distribution at best. In contrast, we show that the cumulative distribution functions of $\hat{Y}_t(n)$ \textit{do} in fact converge pointwise to the expected shifted Gamma distributions. Interestingly, the failure of the nonzero values of $f_{2;n}$ to converge to $g_2$ (Theorem \ref{intro_thm:pmf_no_converge}) suggests that, at the very least in the case $t = 2$, the spacing between points in the support plays a role in the convergence of cumulative distribution functions. By considering the pointwise convergence of characteristic functions in place of moment generating functions, our approach to Theorem \ref{intro_thm:convergence_in_distribution} avoids the obstacles that Griffin--Ono--Tsai encountered in their work. 

Fix $t \geq 2$ and let $k(t) = \frac{t-1}{2}, \theta(t) = \sqrt{\frac{2}{t-1}}$. We show that the sequence of random variables \[\frac{n\pi}{\sqrt{3(t-1)n}} - \frac{t\pi}{\sqrt{3(t-1)n}}\hat{Y}_t(n)\] converges to the random variable $X_{k(t),\theta(t)}$ in distribution.

\begin{thm}\label{intro_thm:convergence_in_distribution}
Assume the notation above. If $t \geq 2$, then the following are true:

\begin{itemize}
    \item[(a)] The sequence $\hat{Y}_t(n)$ satisfies
    \begin{align*}
    \hat{Y}_t(n) \sim \frac{n}{t} - \frac{\sqrt{3(t-1)n}}{\pi t} \cdot X_{k(t), \theta(t)},
    \end{align*}
    and has mean $\mu_t(n) \sim \frac{n}{t} - \frac{(t-1)\sqrt{6n}}{2 \pi t}$, mode $\mathrm{mo}_t(n) \sim \frac{n}{t} - \frac{(t-3)\sqrt{6n}}{ 2 \pi t}$, and variance ${\sigma_t^2(n) \sim \frac{3(t-1)n}{\pi^2t^2}}.$
    \item[(b)] If we let $\xi_{t,n}(x) := \mu_t(n) + \sigma_t(n)x$ and $F_t(k,n)$ denote the cumulative distribution function of $\hat{Y}_t(n)$, then \[\lim_{n \to \infty} F_t(\xi_{t,n}(x),n) = \frac{\gamma\left(\frac{t-1}{2},\sqrt{\frac{t-1}{2}}x + \frac{t-1}{2}\right)}{\Gamma\left(\frac{t-1}{2}\right)},\] where $\gamma(s,t)$ is the lower incomplete gamma function.
\end{itemize}
\end{thm}

Theorem \ref{intro_thm:convergence_in_distribution} extends the results of \cite[Theorem 1.2]{GOT22} to hold for $t = 2, 3$. Thus, our theorem completes the characterization of the limiting distribution of $\{\hat{Y}_t(n)\}$ for all values of $t$.

\begin{figure}[h!]
    \centering
    \includegraphics[width=4in]{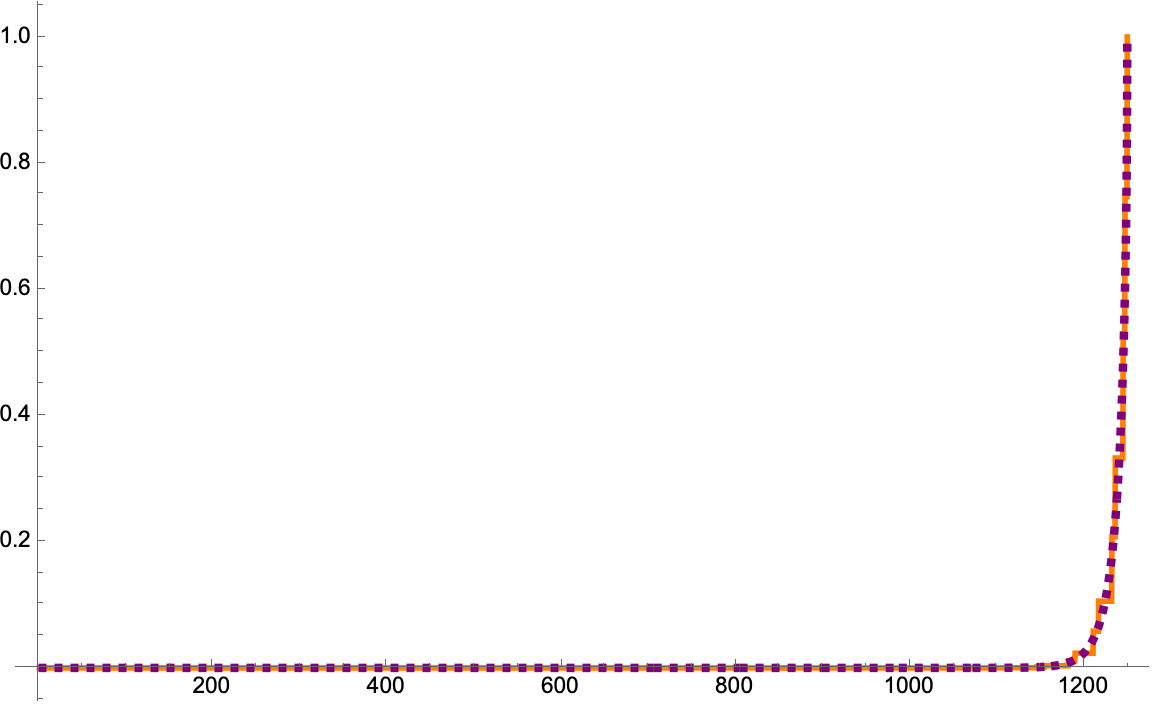}
    \caption{Plot of the cumulative distribution function of $\hat{Y}_2(2500)$ (in orange) and the cumulative distribution function of the shifted random variable specified in Theorem \ref{intro_thm:convergence_in_distribution}(a) with $t = 2, n = 1000$ (in dashed purple).} 
    \label{fig:cdf_2_2500}
\end{figure}

\begin{figure}[h!]
    \centering
    \includegraphics[width=4in]{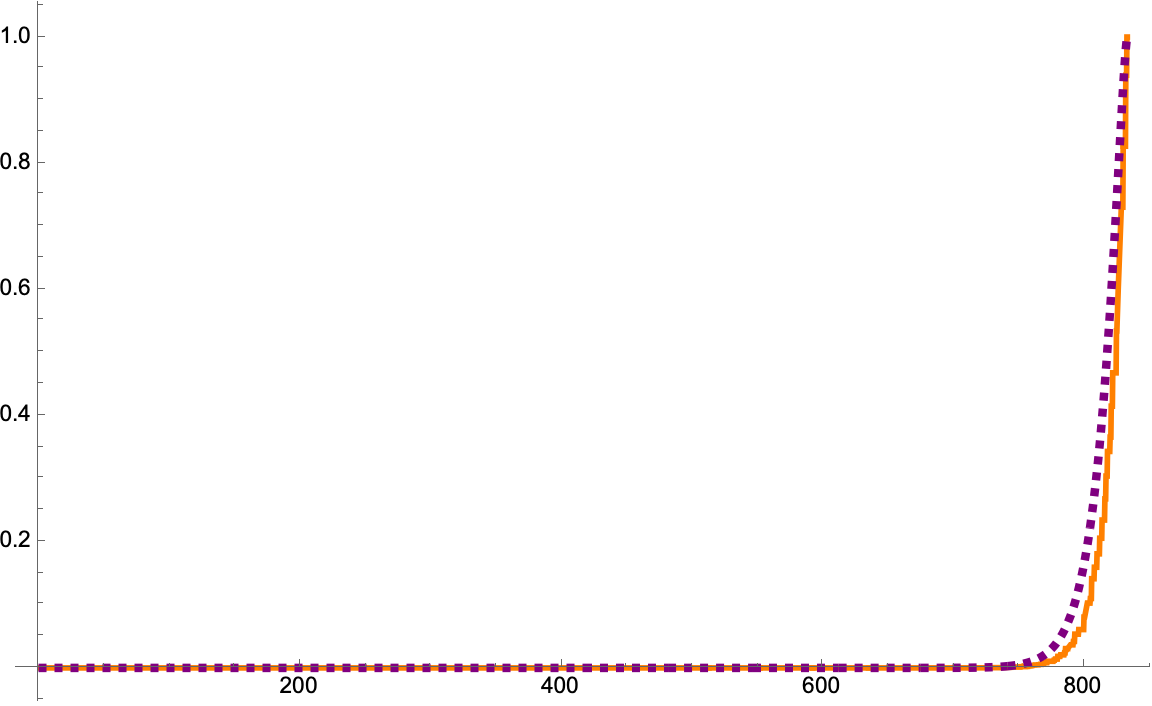}
    \caption{Plot of the cumulative distribution function of $\hat{Y}_3(2500)$ (in orange) and the cumulative distribution function of the random variable specified in Theorem \ref{intro_thm:convergence_in_distribution}(a) with $t = 3, n = 2500$ (in dashed purple).} 
    \label{fig:cdf_3_1000}
\end{figure}

The structure of the article is as follows. In Section \ref{sec:pmfs}, we explicitly describe the supports of the probability mass functions $\hat{Y}_{t}(n)$ for $t = 2, 3$ in Theorems \ref{intro_thm:support_t2} and \ref{intro_thm:support_t3}. Moreover, we show the failure of $f_{2;n}$ to converge to $g_2(x)$ in contrast with the convergence (up to a scalar multiple) of $f_{3;n}(x)$ to $g_3(x)$, thereby proving Theorems \ref{intro_thm:pmf_no_converge} and \ref{intro_thm:pmf_converge_t3}. Finally, in Section \ref{sec:cdfs}, we prove Theorem \ref{intro_thm:convergence_in_distribution} to show that the cumulative distribution functions of $\hat{Y}_t(n)$ for $t = 2, 3$ still converge to the desired shifted Gamma distribution. 

\subsection*{Acknowledgements} The authors would like to thank Wei-Lun Tsai for supervising this project and Ken Ono for his generous support. The authors were participants in the 2022 UVA REU in Number Theory. They are grateful for the support of grants from the National Science Foundation
(DMS-2002265, DMS-2055118, DMS-2147273), the National Security Agency (H98230-22-1-0020), and the Templeton World Charity Foundation.

\section{Probability Mass Functions}\label{sec:pmfs}

In this section, we study the scaled mass functions $f_{t;n}$ corresponding to $\{\hat{Y}_t(n)\}$ for $t = 2, 3$. Recall from Section \ref{sec:intro} that $p_t(m, n)$ denotes the number of partitions of $n$ with $m$ hooks of length divisible by $t$. For each $t \geq 1$, Han \cite{Han10} derived the following multivariate generating function for the quantities $p_t(m,n)$: \begin{align} \label{generating_function}
G_t(x; q) =\prod_{n = 1}^\infty \frac{(1 - q^{tn})^t}{(1 - (xq^t)^n)^t(1 - q^n)} =\sum_{n=0}^\infty \sum_{m=0}^\infty p_t(m, n)x^m q^n.
\end{align} Moreover, recall the notation \[P_t(n;x) := \sum_{m=0}^\infty p_t(m,n)x^m.\] Observe that $P_t(n;x)$ is a polynomial in $x$; we will show that $\deg P_t(n;x) = \lfloor n/t\rfloor$ for each $t = 2,3$. For any positive integers $n$ and $k$, we can describe the probability mass functions of $\hat{Y}_t(n)$ via \begin{align}\label{eqn:probability}\mathbb{P}(\hat{Y}_t(n) = k) = \frac{p_t(k,n)}{p(n)}.\end{align}
Before we proceed, we will need to define, for a fixed nonnegative integer $t$, the $t$-colored partitions of any positive integer $n$. These $t$-colored partitions will play an important role in our characterization of the nonzero values of $p_t(m,n)$ in the cases $t = 2$ and $3$.

\begin{defn}
Fix a nonnegative integer $t$ and a positive integer $n$. A $t$-colored partition of $n$ is a multiset of $t$ (possibly empty) sequences of nonincreasing positive integers such that sum of all integers across all sequences is equal to $n$. 
\end{defn}

Observe that the number of $t$-colored partitions of some positive integer $n$ is given by the coefficient of $y^n$ after formally expanding the product
\begin{align}\label{eqn:tcolored_gf} \prod_{m=1}^\infty \frac{1}{(1-y^m)^t}.\end{align}

We begin in Section \ref{sec:even_hl} with a characterization of the support of $\hat{Y}_t(n)$ for $t=2,3$ and estimate how quickly it vanishes as $n$ grows. As we have seen in Figure \ref{fig:support_looks_smooth}, the nonzero values of the function $f_{2; n}(x)$ appear to approximate a continuous curve -- we determine this curve $h_{2;n}(x)$ in Section \ref{sec:even_hl} using asymptotic formulas of Meinardus and Hardy-Ramanujan. However, we show that even this continuous approximation will not converge pointwise to the probability density function of the Gamma distribution. In Section \ref{sec:three_hl} we determine the support of $f_{3; n}(x)$ and find an asymptotic upper bound on its decay as $n \to \infty$. In contrast to the case of $t = 2$, we show that $f_{2; n}(x)$ approximate multiple continuous curves which converge to scalar multiples of $g_3(x)$. 

\subsection{Even Hook Lengths}\label{sec:even_hl}
In this subsection, we begin with an explicit description of the coefficients $p_2(m,n)$ in terms of the number of two-colored partitions of $m$ and prove Theorem \ref{intro_thm:support_t2}. Then, for all $x \in \mathbb{R}$ such that $f_{2; n}(x) \neq 0$, we estimate $f_{2; n}(x)$ with a continuous function $h_{2;n}(x)$ and use the explicit formula for $h_{2;n}(x)$ to prove Theorem \ref{intro_thm:pmf_no_converge}.

\begin{prop}\label{prop:explicit_coefficients_t_2}
When $t = 2$, the coefficients of $P_2(n;x)$ satisfy \[p_2(m,n) = \begin{cases} a(m) & \text{if $n - 2m$ is a triangular number} \\ 0 & \text{otherwise,}\end{cases}\] where $a(k)$ denotes the number of $2$-colored partitions of $k$. 
\end{prop}

\begin{proof}
By a simple observation, we have \[\prod_{m=1}^\infty \frac{1}{(1-(xq^2)^m)^2} = \sum_{k=0}^\infty a(k)x^kq^{2k},\] where $a(k)$ is the number of $2$-colored partitions of $k$. On the other hand, a classical identity of Jacobi gives that \[\prod_{m=1}^\infty \frac{(1-q^{2m})^2}{(1-q^m)} = \sum_{\ell=1}^\infty q^{\ell(\ell+1)/2}.\] Alternatively, one can obtain the identity above by realizing that the left-hand side is the generating function for the $2$-core partitions, and noting that the $2$-core partitions are precisely the partitions of the form $(\ell,\ell-1,\ldots,2,1)$ for positive integers $\ell$. The desired result follows immediately upon studying the coefficient of $x^mq^n$ in (\ref{generating_function}): \[G_2(x;q) = \prod_{m=1}^\infty \frac{(1-q^{2m})^2}{(1-(xq^2)^m)^2(1-q^m)} = \left(\sum_{k=0}^\infty a(k)x^kq^{2k}\right)\left(\sum_{\ell=1}^\infty q^{\ell(\ell+1)/2}\right). \qedhere\]
\end{proof}

\begin{cor}\label{cor:deg_p_2}
The polynomial $P_2(n,x)$ has degree $\lfloor n/2 \rfloor$. 
\end{cor}

\begin{proof}
By Proposition \ref{prop:explicit_coefficients_t_2}, we see that $\deg P_2(n,x)\leq\lfloor n/2\rfloor$. Note that $n - 2\lfloor n/2 \rfloor = 0$ or $1$, which are both triangular numbers, thus $p_2(\lfloor n/2 \rfloor,n) = a(\lfloor n/2 \rfloor) \neq 0$. 
\end{proof}

As depicted in Figure \ref{fig:t2_sparse}, the distribution of $\hat{Y}_2(100)$ is populated with zeros for $x \in \{0, 1, \dots, 50\}$ (i.e., the support is sparse). We provide asymptotics on the proportion of nonzero coefficients of $P_2(n;x)$ to show that in fact the density of the zeros approaches $1$ in the limit of large $n$.

\begin{cor}\label{cor:vanishing_coefficients_t_2}
As $n \to \infty$, we have \[\frac{\#\{0 \leq m \leq \deg P_2(n,x)\ |\ p_2(m,n) \neq 0\}}{\deg P_2(n,x)} = O(n^{-1/2}).\] In particular, \[\lim_{n \to \infty} \frac{  \#\{0 \leq m \leq \deg P_2(n,x)\ |\ p_2(m,n) = 0\}}{\deg P_2(n,x)} = 1.\]
\end{cor}

\begin{proof}
In light of Proposition \ref{prop:explicit_coefficients_t_2}, we see that \begin{align*}
    \#\{p_2(m,n) \neq 0\} = \#\{0 \leq m \leq \deg P_2(n,x)\ |\ n-2m \text{ is a triangular number}\}.
\end{align*} However, the right-hand side is at most \[\#\{0 \leq k \leq n \ |\ k \text{ is a triangular number}\} = O(n^{1/2}).\] The result follows since Corollary \ref{cor:deg_p_2} gives us that $\deg P_2(n,x) = \lfloor n/2 \rfloor \asymp n$.
\end{proof}

Combining the previous results, we can now prove Theorem \ref{intro_thm:support_t2}.

\begin{proof}[Proof of Theorem \ref{intro_thm:support_t2}]
Observe that (a) follows from Proposition \ref{prop:explicit_coefficients_t_2} and (b) follows from Corollary \ref{cor:vanishing_coefficients_t_2}, so it suffices to prove (c).

Recall that for any positive integer $n$, we defined \[\mathcal{S}_{2;n} := \left\{x \in \R\ \bigg|\ \frac{n}{2} - \frac{\sqrt{3n}}{2\pi}x \in \{0,1,\ldots,\lfloor n/2\rfloor\}\right\}.\] 

Fix $x$ and $n_0$ such that the scaled mass function $f_{2;n_0}$ is defined at $x$. Moreover, for the sake of convenience, we assume that $n_0/2 \in \Z$ and $3n_0$ is square-free (the proof is very similar without these assumptions, albeit more complicated from a notational perspective). Under these assumptions, we see that there must exist some $q \in \Z$ such that \[x = \frac{2q\pi}{\sqrt{3n_0}}.\] Now, for each positive integer $j$, define $n_j := j^2n_0$, so that \[\frac{n_j}{2} - \frac{\sqrt{3n_j}}{2\pi}x = \frac{j^2n_0}{2} - qj \in \{0,1,\ldots, n_j/2\}.\] Note that the nonnegativity of this quantity follows from the fact that $n_0/2 \geq q$ by assumption. Thus, we have produced an infinite sequence $\{n_j\}_{j=1}^\infty$ such that $x \in \mathcal{S}_{2;n_j}$ for all $j$. Now, we claim that there are infinitely many values of $j$ such that $f_{2;n_j}(x) = 0$. 

Proposition \ref{prop:explicit_coefficients_t_2} shows us that  $f_{2;n_j}(x) = 0$ whenever \[n_j - 2\left(\frac{n_j}{2} - \frac{\sqrt{3n_j}}{2\pi}x\right) = 2qj\] is not a triangular number, and indeed, for a fixed integer $q$, there are infinitely many choices of $j$ such that $2qj$ is not a triangular number. Thus, we have shown that $f_{2;n_j}(x) = 0$ for infinitely many values of $j$. Since $g_2(x) \neq 0$ for $x > 0$, we deduce that $f_{2;n_j}(x)$ cannot converge to $g_2(x)$ as $j \to \infty$. 
\end{proof}

Recall from Figure \ref{fig:support_looks_smooth} that the nonzero values of $f_{2;n}(x)$ appear to approximate a continuous function. We will show using Proposition \ref{prop:explicit_coefficients_t_2} that these values are indeed asymptotic to a continuous function.

\begin{lem}\label{lem:meinardus}
For fixed $m$, we have \[\frac{p_2(m,n)}{p(n)} \sim \frac{n\exp\left(\pi\sqrt{\frac{2}{3}}(\sqrt{2m}-\sqrt{n})\right)}{3^{1/4}m^{5/4}}(1+\widetilde{\beta}_2(m)),\] as $n \to \infty$ through suitable values of $n$ (i.e. such that $n -2m$ is a triangular number), where $\widetilde{\beta}_2(m)$ is a constant satisfying $\widetilde{\beta}_2(m) = O_\varepsilon(m^{-1/4+\varepsilon})$ for any $\varepsilon > 0$. 
\end{lem} 

\begin{proof}
A result of Meinardus \cite[Theorem 6.2]{And84} gives us that \[a(m) =
\frac{3^{1/4}}{12 n^{5/4}}\exp\left( 2 \pi \sqrt{\frac{n}{3}} \right) (1 + \widetilde{\beta}_2(m)). \]
Moreover, the famous Hardy-Ramanujan asymptotic formula \cite{HR18} gives us \[p(n) \sim \frac{1}{4n\sqrt{3}}\exp\left(\pi\sqrt{\frac{2n}{3}}\right)\] as $n \to \infty$, whence the desired result follows after using $p_2(m,n) = a(m)$ from Proposition \ref{prop:explicit_coefficients_t_2} and simplifying.
\end{proof}

Despite the apparent continuity of $f_{2;n}(x)$ as $n \to \infty$ suggested by Lemma \ref{lem:meinardus}, we show below in the proof of Theorem \ref{intro_thm:pmf_no_converge} that the continuous approximations $h_{2;n}(x)$ still fail to converge to $g_2(x)$.

\begin{proof}[Proof of Theorem \ref{intro_thm:pmf_no_converge}] Observe that 
\begin{align*}
    f_{2;n}(x) =  \frac{\sqrt{3n}}{2\pi} \, \frac{p_2\left(\frac{n}{2} - \frac{\sqrt{3n}}{2\pi}x, n\right)}{p(n)}.
\end{align*}
As a quick remark, Proposition \ref{prop:explicit_coefficients_t_2} tells us that $f_{2;n}(x) \neq 0$ if and only if $\frac{\sqrt{3n}}{\pi}x$ is a triangular number. In line with our result from Lemma \ref{lem:meinardus}, the choice 
\[
h_{2;n}(x) := \frac{3^{1/4}n^{3/2} \exp \left(\pi \sqrt{\frac{2}{3}} \left( \sqrt{n - \frac{\sqrt{3n}}{\pi}x} - \sqrt{n}\right) \right)}{2 \pi\left( \frac{n}{2} - \frac{\sqrt{3n}}{2\pi}x \right)^{5/4}}.
\] satisfies the conditions in the theorem statement. On the other hand, for any fixed $x \in \R$,
\[
\exp \left(\pi \sqrt{\frac{2}{3}} \left( \sqrt{n - \frac{\sqrt{3n}}{\pi}x} - \sqrt{n}\right) \right) = \Omega(1)
\] 
as $n \to \infty$, so it follows that $h_{2;n}(x) \to \infty$ as $n \to \infty$, concluding the proof.
\end{proof}

\subsection{Hook Lengths Divisible by Three}\label{sec:three_hl} In this subsection, we find an explicit formula for the coefficients of the polynomial $P_3(n;x)$ and prove Theorems \ref{intro_thm:support_t3} and \ref{intro_thm:pmf_converge_t3}.

\begin{prop}\label{prop:explicit_coefficients_t_3}
In the case $t = 3$, the coefficients are given by \[p_3(m,n) = b(m)c(n-3m)\] where $b(k)$ denotes the number of $3$-colored partitions of $k$, \[c(k) := \sum_{d |3k+1} \left(\frac{d}{3}\right),\] and $\left(\frac{a}{p}\right)$ denotes the Legendre symbol.
\end{prop}

\begin{proof}
Similar to the proof of Proposition \ref{prop:explicit_coefficients_t_2}, first notice that \[\prod_{m=1}^\infty \frac{1}{(1-(xq^3)^m)^3} = \sum_{k=0}^\infty b(k)x^kq^{3k},\] where $b(k)$ is the number of $3$-colored partitions of $n$. On the other hand, Han--Ono \cite{HO11} showed that \[\prod_{m=1}^\infty \frac{(1-q^{3m})^3}{1-q^m} = \sum_{k=0}^\infty c(k)q^k,\] where $c(k)$ is defined as in the statement of the proposition. We obtain the proposition by computing the coefficient of $x^mq^n$ in (\ref{generating_function}): \[G_3(x;q) = \prod_{m=1}^\infty \frac{(1-q^{3m})^3}{(1-(xq^3)^m)^3(1-q^m)} = \left(\sum_{k=0}^\infty b(k)x^kq^{3k}\right)\left(\sum_{k=0}^\infty c(k)q^k\right). \qedhere\]
\end{proof}

The properties of the coefficients $c(k)$ were studied in depth in \cite{HO11}. Their results allow us to explicitly characterize the nonzero coefficients and compute the degree of $P_3(n;x)$. 

\begin{cor}\label{cor:zero_coeffs_t_3}
We have $p_3(m,n) = 0$ if and only if $c(n-3m) = 0$. In particular, $p_3(m,n) = 0$ if and only if there exists a prime $r \equiv 2\pmod{3}$ such that $\operatorname{ord}_r(3(n-3m)+1)$ is odd.
\end{cor}

\begin{proof}
This corollary is an immediate consequence of \cite[Theorem 1.1]{HO11} and the fact that $b(k) \neq 0$ for any nonnegative integer $k$. 
\end{proof}

\begin{cor} \label{cor:deg_P_3}
The polynomial $P_3(n, x)$ has degree $\lfloor n/3 \rfloor$.
\end{cor}

\begin{proof}
Observe that $n - 3\lfloor n/3 \rfloor$ is either $0$, $1$, or $2$, and $c(0) = 1$, $c(1) = 1$, $c(2) = 2$. 
\end{proof}

\begin{cor}
If $3(n - 3m)+1$ is prime, then $p_3(m,n) = 2b(m)$. 
\end{cor}

\begin{proof}
If $3(n - 3m)+1$ is prime, then \[c(n-3m) = \left(\frac{1}{3}\right) + \left(\frac{3(n-3m)+1}{3}\right) = 2. \qedhere\] 
\end{proof}

We now seek an asymptotic upper bound on the proportion of nonzero coefficients in $P_3(n;x)$ as $n \to \infty$. Before proceeding, we establish some notation from \cite{HO11}. Consider the series \[D(q) := \prod_{m=1}^\infty (1-q^m)^8 :=  \sum_{k=0}^\infty d(k)q^k.\] To further study the coefficients $d(k)$, Han--Ono \cite{HO11} noted that the \textit{renormalized series} \[\mathcal{D}(q) := \sum_{n = 1}^\infty d^*(n) q^n := \sum_{n = 0}^\infty d(n)q^{3n + 1}\] arises as the $q$-expansion of a modular form belonging to $S_4(\Gamma_0(9))$, the space of weight $4$ cusp forms on $\Gamma_0(9)$. They used the fact that $\dim S_4(\Gamma_0(9)) = 1$ to show that $\mathcal{D}(q)$ is in fact a normalized Hecke eigenform, and thus deduce general multiplicative properties of its Fourier coefficients $d(k)$. Moreover, by relating the renormalized series \[C(q) := \sum_{k=0}^\infty c(k)q^{3k+1}\] the norm form on the ring of integers of the imaginary quadratic field $\Q(\sqrt{-3})$, they derived the formula for $c(k)$ in the statement of Proposition \ref{prop:explicit_coefficients_t_3} and in turn showed that $c(k) = 0$ if and only $d(k) = 0$ for all nonnegative integers $k$. Thus, in light of Corollary \ref{cor:zero_coeffs_t_3}, we can quantify the rate at which the support of $\hat{Y}_3(n)$ vanishes by exploiting the multiplicative properties of $d(k)$. 

Recall that a set $E \subset \mathbb{Z}^{+}$ is \textit{multiplicative} if for relatively prime integers $n_1, n_2 \in \mathbb{Z}^{+}$, we have $n_1 n_2 \in E$ if and only if $n_1 \in E$ or $n_2 \in E$. We use the aforementioned properties of $d(k)$ in tandem with \cite[Theorem 2.4(a)]{Ser75} on the sizes of complements of multiplicative sets to estimate the proportion of nonzero coefficients of $P_3(n,x)$. In the same spirit as Corollary \ref{cor:vanishing_coefficients_t_2}, the following result allows us to say that ``almost every" coefficient of $P_{3}(m,n)$ is equal to zero. 

\begin{cor}\label{cor:vanishing_coefficients_t_3}
As $n \to \infty$, \[\frac{\#\{0 \leq m \leq \deg P_3(n, x) \ | \ p_3(m,n) \neq 0\}}{\deg P_3(n,x)} = O\left(\frac{1}{\sqrt{\log(n)}}\right).\]

In particular, \[\lim_{n \to \infty} \frac{\#\{0 \leq m \leq \deg P_3(n, x) \ | \ p_3(m,n) = 0\}}{\deg P_3(n,x)} = 1.\]
\end{cor}

\begin{proof}
First consider the set $E := \{m \in \mathbb{Z}^{+} \ | \ d^*(m) = 0\},$ which consists of the integers $m$ where the renormalized coefficients $d^*(m)$ vanish. Moreover, consider the set $P$ consisting of all primes congruent to $2 \pmod{3}$. By \cite[Corollary 2.2]{HO11}, the set $E$ is multiplicative, and $P$ is precisely the set of primes in $E$. Due to a classical result of Dirichlet, the set $P$ has natural density $1/2$ in the natural numbers. Then, \cite[Theorem 2.4]{Ser75} tells us that for any $z > 0$, \[\overline{E}(z) := \# \{0 \leq \ell \leq z \ | \ \ell \not\in E \} = \#\{0 \leq \ell \leq z \ |\ d^*(\ell) \neq 0\} = O\left(\frac{z}{\sqrt{\log(z)}}\right).\] Thus, using the fact that $c(k) = 0$ if and only if $d(k) = 0$ (see \cite[Theorem 1.1]{HO11}), we deduce that
\begin{align*}
\# \{0 \leq m \leq \lfloor n/3 \rfloor \ | \ p_3(m, n) \neq 0 \} & \leq \# \{0 \leq k \leq n \ | \ c(k) \neq 0 \} \\[5pt] & = \# \{0 \leq k \leq n \ | \ d(k) \neq 0 \} \\[5pt]
& = \# \{0 \leq k \leq n \ | \ d^*(3k + 1) \neq 0 \} \\[5pt]
& \leq \# \{0 \leq l \leq 3n + 1 \ | \ d^*(l) \neq 0 \} \\[5pt]
& = O\left(\frac{n}{\sqrt{\log(n)}}\right).
\end{align*}
The desired results follow after dividing by $\deg P_3(n,x) = \lfloor n/3\rfloor \asymp n$. \qedhere
\end{proof}

\begin{proof}[Proof of Theorem \ref{intro_thm:support_t3}]
Combining the results in this subsection completes the proof of Theorem \ref{intro_thm:support_t3}. In particular, note that Theorem \ref{intro_thm:support_t3}(a) follows from Proposition \ref{cor:zero_coeffs_t_3}, while Theorem \ref{intro_thm:support_t3}(b) follows from Corollary \ref{cor:vanishing_coefficients_t_3}. Finally, Theorem \ref{intro_thm:support_t3}(c) follows from a similar argument as in the proof of Theorem \ref{intro_thm:support_t2}(c) in Section \ref{sec:even_hl}.
\end{proof}

In preparation for the proof of Theorem \ref{intro_thm:pmf_converge_t3}, we show that the values of $f_{3; n}(x)$ restricted to its support approximate multiple continuous curves using similar techniques as in Lemma \ref{lem:meinardus}.

\begin{lem}\label{lem:meinardus_t3}
We have \[\frac{p_3(m, n)}{p(n)} \sim  \frac{n\sqrt{3} \exp\left(\pi \sqrt{\frac{2}{3}} (\sqrt{3m} - \sqrt{n})\right)}{2^{3/2}m^{3/2}} \widetilde{\alpha}(m,n)(1 + \widetilde{\beta}_3(m))\] as $n \to \infty$ for all $m$ such that $f_{3; n}(m) \neq 0$, where $\widetilde{\beta}_3(m) = O_\varepsilon(m^{-1/4 + \varepsilon})$ for any $\varepsilon > 0$.
\end{lem}

\begin{proof} 
Recall from Proposition \ref{prop:explicit_coefficients_t_3} that $p_3(m,n) = b(m)c(n-3m)$, and set $\widetilde{\alpha}(m,n) := c(n-3m)$. A result of Meinardus \cite[Theorem 6.2]{And84} gives us that \[b(m) = \frac{1}{2^{7/2}m^{3/2}} \exp (\pi \sqrt{2m})(1 + \widetilde{\beta}_3(m)).\] As in the proof of Lemma \ref{lem:meinardus}, the desired result follows after applying the Hardy-Ramanujan asymptotic formula \cite{HR18} for $p(n)$ and simplifying. 
\end{proof}

\begin{xrem} 
We make a few comments on the values of the integers $\alpha(m,n) = c(n - 3m)$. First, for any integer $k$, note that the character sum \[c(k) = \sum_{d|3k+1} \left(\frac{d}{3}\right)\] is odd if and only if $k$ is a perfect square, as the Legendre symbol is completely multiplicative. As the perfect squares have natural density zero in the integers, it follows that $100\%$ of the scalars $\alpha(m,n)$ are even numbers in the limit of large $n$. Moreover, if we factor $k = r_1^2r_2,$ where $r_1$ is divisible only by primes equivalent to $2 \pmod{3}$ and $r_2$ is divisible only by primes equivalent to $1 \pmod{3}$, then $c(k)$ is precisely the number of divisors of $r_2$. In light of this fact, one can study the frequency with which $\alpha(m,n) = c(n-3m) = \ell$ for any positive integer $\ell$ by applying results on the density of the prime numbers.
\end{xrem}

In the case $t = 3$, we have so far shown that the nonzero points of $f_{3;n}(m)$ lie on integer multiples of the continuous function $h_{3;n}(m)$. When $t = 2$, Theorem \ref{intro_thm:pmf_no_converge} tells us that the continuous approximations $h_{2; n}(m)$ fail to converge to $g_2(m)$ as $n \to \infty$. In contrast, we now prove that $h_{3; n}(m)$ in fact converges to a scalar multiple of the density function $g_3(m)$. 

\begin{proof}[Proof of Theorem \ref{intro_thm:pmf_converge_t3}]
Observe that 
\[
f_{3;n}(x) = \frac{\sqrt{6n}}{3\pi} \, \frac{p_3\left( \frac{n}{3} - \frac{\sqrt{6n}}{3\pi}x, n \right)}{p(n)}.
\]
Corollary \ref{cor:zero_coeffs_t_3} tells us that $f_{3;n}(x) \neq 0$ if and only if $c\left(\frac{\sqrt{6n}}{\pi}x\right) = 0$. By Lemma \ref{lem:meinardus_t3}, the choice of 
\[
h_{3;n}(x) :=  \frac{3\sqrt{3}n^{3/2}\exp\left(\pi\sqrt{\frac{2}{3}}\left(\sqrt{n-\frac{\sqrt{6n}}{\pi}x} - \sqrt{n}\right)\right)}{2\pi\left(n - \frac{\sqrt{6n}}{\pi}x\right)^{3/2}}
\]
and $\alpha(x,n) := \widetilde{\alpha}\left(\frac{n}{3} - \frac{\sqrt{6n}}{3\pi}x, n\right)$ satisfies the conditions of the theorem. 

For statement (b), recall that $g_3(x) = e^{-x}$. For all $x \in \mathbb{R}$, applying the Taylor series expansion (in $x$) of \[n^{1/2} \left(1 - \frac{\sqrt{6}}{\sqrt{n} \pi}x\right)^{1/2}\] gives us
\begin{align*}
\lim_{n \to \infty} h_{3; n}(x) = \lim_{n \to \infty} \frac{3\sqrt{3}n^{3/2}\exp\left(\pi\sqrt{\frac{2}{3}}\left(\sqrt{n-\frac{\sqrt{6n}}{\pi}x} - \sqrt{n}\right)\right)}{2\pi\left(n - \frac{\sqrt{6n}}{\pi}x\right)^{3/2}} = \frac{3\sqrt{3}}{2 \pi} g_3(x), 
\end{align*} and the desired result follows.
\end{proof}

\section{Convergence in Distribution}\label{sec:cdfs}
In this section, we show that the cumulative distribution function of $\{\hat{Y}_t(n)\}$ converges to a shifted Gamma distribution for $t = 2, 3$, thereby fully resolving the question of limiting distributions of $\{\hat{Y}_t(n)\}$ for all $t \geq 2$. We proceed by using the generating function (\ref{generating_function}) to estimate values of the polynomials $P_t(n,x)$ using the saddle point method. We will then write the characteristic functions of the random variables $\hat{Y}_t(n)$ in terms of these polynomials and use these estimates to prove the desired pointwise convergence of the characteristic functions. 

The following proposition is our main tool for obtaining the estimates necessary to prove Theorem \ref{intro_thm:convergence_in_distribution}. Griffin--Ono--Tsai proved a similar result the case where $\alpha$ takes on strictly real values \cite[Proposition 2.2]{GOT22}, but we show that the result can be extended to the case where $\alpha$ takes on purely imaginary values. The proof of \cite[Proposition 2.2]{GOT22} can be readily adapted to prove our modified proposition, so our treatment of its proof will be rather terse. 

\begin{prop} \label{prop:formula_P(n, x)}
Fix $t = 2, 3$ and let $\alpha \in i\R$ be purely imaginary. Then, the following asymptotic estimate holds in the limit of large $n$: \[P_t\left(n,x_n\right) = \frac{1}{2^{7/4}3^{1/4}n}\sqrt{\frac{1}{\sqrt{6}} + \frac{\alpha}{\pi t}}\left(\frac{\pi t}{\pi t + \alpha\sqrt{6}}\right)^{t/2} \exp\left(\pi\sqrt{n}\left(\sqrt{\frac{2}{3}}+\frac{\alpha}{\pi t}\right)\right) \cdot (1 + O_\alpha(n^{-1/7})),\] where $x_n := \exp\left(\alpha/\sqrt{n}\right)$ for each positive integer $n$.
\end{prop}

\begin{proof} 
Applying the Cauchy residue theorem to (\ref{generating_function}) gives us \[P_t(n,x_n) = \frac{1}{2\pi} \int_{-\pi}^\pi (ze^{i\tau})^{-n}G_t(x_n;ze^{i\tau})\ d\tau = \frac{1}{2\pi} \int_{-\pi}^{\pi} \exp(g_t(x_n;ze^{i\tau}))\ d\tau,\] where \begin{align*}g_t(x_n;w) &:= \mathrm{Log}(w^{-n}G_t(x_n;w))\\[10pt] &= -n\log(w) + \sum_{m=1}^\infty t\log(1-w^{tm}) - \sum_{m=1}^\infty t\log(1-(x_nw^t)^m) - \sum_{m=1}^\infty \log(1-w^m),\end{align*} provided $0 < |w| < 1$. To apply the saddle point method, we seek $\kappa\in \C$ with $0 < |\kappa| < 1$ such that \[g_t(x_n,\kappa) \neq 0, \quad \frac{d}{dw}\bigg|_{w=\kappa}[g_t(x_n;w)] = 0,\quad \frac{d^2}{dw^2}\bigg|_{w=\kappa}[g_t(x_n;w)]\neq 0.\] We set the derivative of $g_t(x_n;w)$ with respect to $w$ equal to zero and follow a line of reasoning identical to that in the proof of \cite[Proposition 2.2]{GOT22}. In particular, 
\begin{align*}
    \frac{d}{dw}g(x_n,w) = -\frac{m}{w} - \sum_{j=1}^\infty \frac{t^2j}{1 - w^{tj}}w^{tj-1} + \sum_{j=1}^\infty \frac{t^2j x_n^j}{1 - (x_nw^{j})^j} w^{tj-1} + \sum_{j=1}^\infty \frac{jw^{j-1}}{1 - w^{j}}
\end{align*}
gives 
\begin{align*}
    n = \sum_{j=1}^\infty \frac{t^2j}{w^{-tj} - 1} + \sum_{j=1}^\infty \frac{t^2j x_n^j}{w^{-tj} - x_n^j} + \sum_{j=1}^\infty \frac{j}{w^{-j} -1}.
\end{align*}
By the definition of $x_n$ and
\begin{align}
    \label{2.8analogue} \sum_{j=1}^\infty \frac{j}{e^{j \alpha} - 1} = \frac{\pi^2}{6\alpha^2} - \frac{1}{2\alpha} + O(1), 
\end{align}
we obtain the saddle point $\kappa = e^{-\rho_n}$, where \[\rho_n := \left(\frac{\pi}{\sqrt{6}} + \frac{\alpha}{t}\right)n^{-1/2} + O_\alpha(n^{-1}).\] In particular, notice that (for sufficiently large $n$) \[|\kappa| = |e^{-\rho_n}| = |e^{-\operatorname{Re}(\rho_n)}| < 1.\]
With this value of $\kappa$, we see that (again following the arguments in \cite[Proposition 2.2]{GOT22}) \begin{align*}
    g_t(x_n,\kappa) &= n\rho_n + \sum_{m=1}^\infty t\log(1-e^{-tm\rho_n}) - t\sum_{m=1}^\infty \log(1-x_n^me^{-tm\rho_n}) - \sum_{m=1}^\infty \log(1-e^{-m\rho_n}) \\[10pt] &=
    n\rho_n + \sum_{m=1}^\infty t\log(1-e^{-tm\rho_n}) - t\sum_{m=1}^\infty \log(1-e^{-m(t\rho_n - \alpha/\sqrt{n})}) - \sum_{m=1}^\infty \log(1-e^{-m\rho_n}).
\end{align*}
By applying Euler-Maclaurin summation in the style of \cite[Lemma 2.3]{GOT22}, we have \[\sum_{m=1}^\infty \log(1-e^{-j\gamma}) = -\frac{\pi^2}{6\gamma} - \frac{1}{2}\mathrm{log}\left(\frac{\gamma}{2\pi}\right) + O(\gamma)\] for any $\gamma \in \C$ with positive real part. Using this equality, we obtain the same estimate for $g_t(x_n,\kappa)$ as in the proof of \cite[Proposition 2.2]{GOT22}. In the same way as we have outlined so far, the rest of the proof follows \textit{mutatis mutandis} to the proof of \cite[Proposition 2.2]{GOT22}, using the facts that $\rho_n$ has positive real part and multiplication by $x_n$ does not affect absolute value. 
\end{proof}

With this proposition in place, we can now prove the convergence in distribution of $\{\hat{Y}_t(n)\}$ to a shifted Gamma distribution. Due to the following continuity theorem by L\'evy, it suffices to prove pointwise convergence of characteristic functions. 

\begin{thm}[L\'evy] \cite[Section 18.1]{Wil91}
Let $(X_n)$ be a sequence of random variables with characteristic functions given by $\varphi_n(r):= \mathbb{E}[e^{iX_nr}]$ and let $X$ be a random variable with characteristic function $\varphi(r) := \mathbb{E}[e^{iXr}]$. Suppose that
\[
\varphi(r) = \lim_{n \to \infty} \varphi_n(r)
\]
for all $r \in \R$ and that $\varphi$ is continuous at $r = 0$. Then $X_n$ converges to $X$ in distribution. 
\end{thm}

The following proof of Theorem \ref{intro_thm:convergence_in_distribution} is similar to the proof of \cite[Theorem 1.2]{GOT22} in the case of $t \geq 4$. The methods of Griffin--Ono--Tsai do not apply to the cases $t = 2, 3$ since the moment generating functions of the corresponding Gamma distributions do not exist, but replacing moment generating functions with characteristic functions allows us to extend the result to all $t \geq 2$. Recall the notation $k(t) = \frac{t-1}{2}$ and $\theta(t) = \sqrt{\frac{2}{t-1}}$ from Section \ref{sec:intro}.

\begin{proof}[Proof of Theorem \ref{intro_thm:convergence_in_distribution}]
The characteristic function of the Gamma distribution $X_{k, \theta}$ is given by $$\varphi(X_{k, \theta}, r) =\left(\frac{1}{1 - ir\theta}\right)^k,$$ where $X_{k, \theta}$ has mean $\mu_{k, \theta} = k \theta$, mode $\text{mo}_{k, \theta} = (k - 1)\theta$ and variance $\sigma_{k, \theta}^2 = k \theta^2$. For $a, b \in \mathbb{R}$, the shifted Gamma distribution $a X_{k, \theta} + b$ has characteristic function $$\varphi(a X_{k, \theta} + b, r) = \frac{e^{br}}{(1 - i \theta a r)^k},$$ with mean $a k \theta + b$, mode $a(k - 1)\theta + b$ and variance $a^2 k \theta^2$. We proceed by comparing $\hat{Y}_t(n)$ with $a X_{k, \theta} + b$, where $k = \frac{t - 1}{2}, \theta = \sqrt{\frac{2}{t - 1}}, a = -1$, and $b = \sqrt{2(t - 1)}/2.$ 

By (\ref{eqn:probability}), the characteristic functions of $\hat{Y}_t(n)$ are given by $$\varphi(\hat{Y}_t(n), r) = \frac{1}{p(n)} \sum_{m = 0}^\infty p_t(m; n) e^{\frac{i(m - \mu_t(n))r}{\sigma_t(n)}}$$ for each $n \geq 1$. Due to L\'evy's continuity theorem and the above remarks, it suffices to show $$\lim_{n\to \infty} \varphi(\hat{Y}_t(n); r) = \frac{e^{br}}{(1 - i \theta a r)^k}$$ for any $r \in \R$. We first evaluate $P_t(n; x)$ at $x = e^{\frac{ir}{\sigma_t(n)}}$ to obtain $$\varphi(\hat{Y}_t(n), r) = \frac{P_t(n; e^{\frac{ir}{\sigma_t(n)}})}{p(n)}e^{-i\frac{\mu_t(n)}{\sigma_t(n)}r}.$$ By Proposition \ref{prop:formula_P(n, x)} with $\alpha = i\pi t r / \sqrt{3(t - 1)}$ and $x_n = e^{\frac{\alpha}{\sqrt{n}}}$, we find that

\begin{align*}
\varphi(\hat{Y}_t(n), r) & = \frac{(2^{\frac74}3^{\frac14}n)^{-1}\sqrt{\frac{1}{\sqrt{6}} + \frac{ir}{\sqrt{3(t - 1)}}} \cdot (1 + \sqrt{\frac{2}{t - 1}}ir)^{-\frac{t}{2}} \cdot (1 + O_r(n^{-\frac17}))}{(4 \sqrt{3}n)^{-1} \cdot (1 + O(n^{-\frac17}))} \cdot e^{\frac{n}{t \sigma_t(n)}ir - \frac{\mu_t(n)}{\sigma_t(n)}ir} \\
& = \frac{e^{\frac{\sqrt{2(t - 1)}}{2}ir}}{(1 + \sqrt{\frac{2}{t - 1}}ir)^{\frac{t - 1}{2}}} \cdot (1 + O_r(n^{-\frac17})).
\end{align*}

Taking the limit as $n \to \infty$, L\'evy's continuity theorem gives us \[\hat{Y}_t(n) \sim \sigma_t(n) (a X_{k, \theta} + b) + \mu_t(n).\] We also obtain claim (2) since the random variable $X_{k, \theta}$ has cumulative distribution function $F_{k, \theta}(x) = \gamma(k, \frac{x}{\theta})/\Gamma(k)$ \cite[II.2]{Fel71}, where $\gamma(s, t)$ is the lower incomplete Gamma function.
\end{proof}

\section{Conclusion}\label{sec:conj}

Theorems \ref{intro_thm:support_t2}(c) and \ref{intro_thm:support_t3}(c) describe respectively how the many vanishing terms in the distributions of $\hat{Y}_t(n)$ prevent the convergence of the scaled mass functions $f_{t;n}$ for $t = 2,3$ to $g_t(x)$, for various notions of convergence. In the case of $t \geq 4$, however, none of the coefficients of $P_t(n;x)$ vanish \cite{granville_ono_1996}, even in the limit of large $n$. While the work of \cite{GOT22} establishes convergence of the cumulative distribution functions of $\{\hat{Y}_t(n)\}$ to a shifted Gamma distribution for $t \geq 4,$ the corresponding probability mass functions have yet to be studied. Thus, comparison with Theorems \ref{intro_thm:support_t2}(c), \ref{intro_thm:support_t3}(c) leads to the natural question of whether the functions $f_{t;n}(x)$ converge to $g_t$ for $t \geq 4$. Indeed, there is experimental evidence in support of the positive -- see, for instance, Figure \ref{fig:conjecturet11} for $t = 11$.

\begin{figure}[H]
    \centering
    \includegraphics[width=4in]{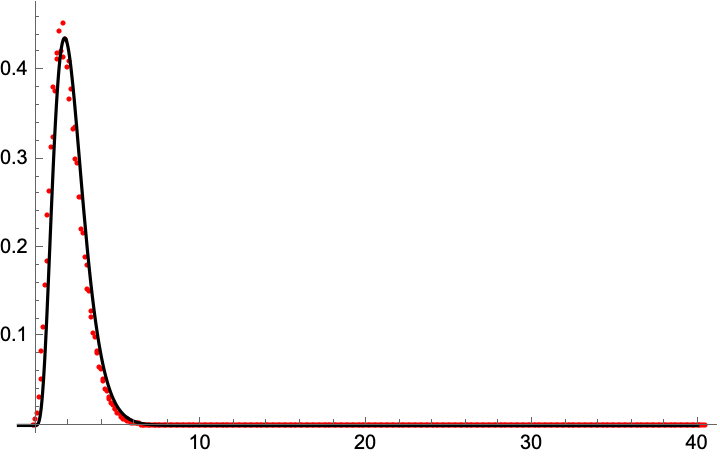}
    \caption{Plot of the scaled mass function $f_{11;5000}(x)$ (in red) with the probability density function $g_{11}(x)$.}
    \label{fig:conjecturet11}
\end{figure}

\begin{conj} For $t \geq 4$, we have $f_{t;n}(x)\to g_t(x)$ as $n \to \infty$ for all $x > 0$.
\end{conj}

\nocite{*}
\bibliography{hooklengthdistributionst23}

\begin{thebibliography}{10}

\bibitem{And84}
George Andrews.
\newblock {\em The {Theory} of {Partitions}}.
\newblock Cambridge University Press, Cambridge, 1984.

\bibitem{Fel71}
William Feller.
\newblock {\em An {Introduction} to {Probability} {Theory} and {Its}
  {Applications}}, volume~2.
\newblock John Wiley \& Sons, Inc., 2nd edition, 1971.

\bibitem{FRT54}
J.~Sutherland Frame, Gilbert de~Beauregard~Robinson, and Robert~M. Thrall.
\newblock The {Hook} {Graphs} of the {Symmetric} {Group}.
\newblock {\em Canadian Journal of Mathematics}, 6:316 -- 324, 1954.

\bibitem{granville_ono_1996}
Andrew Granville and Ken Ono.
\newblock Defect zero blocks for finite simple groups.
\newblock {\em Transactions of the American Mathematical Society},
  348(1):331–347, 1996.

\bibitem{GOT22}
Michael Griffin, Ken Ono, and Wei-Lun Tsai.
\newblock Distributions of {Hook} {Lengths} in {Integer} {Partitions}, 2022.
\newblock https://arxiv.org/abs/2201.06630.

\bibitem{Han10}
Guo-Niu Han.
\newblock The {Nekrasov-Okounkov} {Hook} {Length} {Formula}: {Refinement,}
  {Elementary} {Proof}, {Extension} and {Applications}.
\newblock {\em Annales de l'Institut Fourier}, 60:1 -- 29, 2010.

\bibitem{HO11}
Guo-Niu Han and Ken Ono.
\newblock Hook {Lengths} and 3-{Cores}.
\newblock {\em Annals of Combinatorics}, 15(2):305–312, 2011.

\bibitem{HR18}
Godfrey~Harold Hardy and Srinivasa Ramanujan.
\newblock Asymptotic {Formulae} in {Combinatory} {Analysis}.
\newblock {\em Proceedings of the London Mathematical Society}, s2-17:75 --
  115, 1918.

\bibitem{JK84}
Gordon James and Adalbert Kerber.
\newblock {\em The {Representation} {Theory} of the {Symmetric} {Group}}.
\newblock Cambridge University Press, 1984.

\bibitem{NO06}
Nikita~A. Nekrasov and Andrei Okounkov.
\newblock {\em Seiberg-{Witten} {Theory} and {Random} {Partitions}}, pages 525
  -- 596.
\newblock Birkh{\"a}user Boston, Boston, MA, 2006.

\bibitem{Ser75}
Jean-Pierre Serre.
\newblock Divisibilit\'e de {Certaines} {Fonctions} {Arithm\'etiques}.
\newblock {\em S\'eminaire Delange-Pisot-Poitou. Th\'eorie des Nombres},
  16(1):1 -- 28, 1975.

\bibitem{Wil91}
David Williams.
\newblock {\em {Probability} with {Martingales}}.
\newblock Cambridge University Press, 1st edition, 1991.

\end{thebibliography}
\bibliographystyle{plain}

\end{document}